\documentclass{article}
\usepackage{amsmath,amsthm}
\usepackage{amssymb}
\usepackage[utf8]{inputenc}
\usepackage[english]{babel}
\usepackage{enumitem}
\usepackage{mathtools}
\usepackage{verbatim}
\usepackage{parskip}
\usepackage{float}
\usepackage{tikz}
\usepackage{authblk}

\newtheorem{theorem}{Theorem}[section]

\newtheorem{lemma}[theorem]{Lemma}
\newtheorem{definition}{Definition}
\newtheorem{conjecture}{Conjecture}

\usepackage[margin=1.1in]{geometry}
\usepackage{setspace}
\title{Edge Disjoint Caterpillar Realizations}
\author[$1,2,3$]{Istv\'an Mikl\'os} 
\author[$1,4$]{Geneva Schlafly} 
\author[$1,5$]{Yuheng Wang}
\author[$1,6$]{Zhangyang Wei}
\affil[$1$]{Budapest Semesters in Mathematics\\ 1071 Budapest, Bethelen G\'abor t\'er 2\\ Hungary}
\affil[$2$]{R\'enyi institute\\ 1053 Budapest, Re\'altanoda u. 13-15\\ Hungary}
\affil[$3$]{SZTAKI\\ 1111 Budapest, L\'agym\'anyosi u. 11\\ Hungary}
\affil[$4$]{University of California, Santa Barbara, CA 93106, USA}
\affil[$5$]{Carleton College, Northfield, MN 55057, USA}
\affil[$6$]{Carnegie Mellon University,  5000 Forbes Ave, Pittsburgh, PA 15213, USA}

\begin{document}

\maketitle

\begin{abstract}
    In this paper, we consider the edge disjoint caterpillar realizations of tree degree sequences. We give the necessary and sufficient conditions when two tree degree sequences have edge disjoint caterpillar realizations. We conjecture that an arbitrary number of tree degree sequences have edge disjoint realizations if every vertex is a leaf in at most one tree. We prove that the conjecture is true if the number of tree degree sequences is at most $4$. We also prove that the conjecture is true if $n \ge \max\{22k-11, 396\}$, where $n$ is the number of vertices and $k$ is the number of tree degree sequences.
\end{abstract}

\section{Introduction}
A degree sequence $D=d_1,d_2,\ldots d_n$ is a series of non-negative integers. A degree sequence is graphical if there is a vertex labeled graph $G$ in which the degrees of the vertices are exactly $D$. Such graph $G$ is called a realization of $D$. The color degree matrix problem, also known as an edge disjoint realization, edge packing or graph factorization problem, is the following: Given a $k \times n$ degree matrix $\mathbf{D}=\{\{d_{1,1},d_{1,2},\ldots d_{1,n} \},\{d_{2,1},d_{2,2},…d_{2,n} \},\ldots \{d_{k,1},d_{k,2},…d_{k,n}\}\}$, in which each row of the matrix is a degree sequence, decide if there is an ensemble of edge disjoint realizations of the degree sequences. Such a set of edge disjoint graphs is called a realization of the degree matrix. A realization can also be presented as an edge colored simple graph, in which the edges with a given color form a realization of the degree sequence in a given row of the color degree matrix.

The existence problem in general is a hard computational problem for any $k \ge 2$ \cite{bentzetal2009}. However, it is easy for some special cases. One special case is when the degree matrix is very sparse, the total sum of the degrees is at most $2n-1$, where $n$ is the number of vertices. In that case, necessary and sufficient conditions exist for realizing a colored degree matrix with a colored forest \cite{h-mcdc}. Another interesting case is when each degree sequence is a degree sequence of a tree. We will call these tree degree sequences. 
Kundu proved that two tree degree sequences have edge disjoint tree realizations if and only if the sum of the degree sequences is graphical \cite{kundu2trees}. He also proved that a similar statement is not true for three degree sequences. He gave an example of three tree degree sequences such that the sum of any two of them is graphical and the sum of all three degree sequences is graphical, but the degree sequences do not have edge disjoint tree realizations \cite{kundu3trees}. On the other hand, he proved that three tree degree sequences always have edge disjoint tree realizations if the minimum sum of the degrees on each vertex is at least $5$ \cite{kundu3trees}. This condition includes the case when each vertex is a leaf in at most one of the trees. We conjecture that a degree matrix always has edge disjoint caterpillar realizations if each row is a tree degree sequence and each vertex is a leaf in at most one of the trees.

In this paper we prove that this conjecture holds when the number of degree sequences is at most $4$ or the number of vertices is at least $\max\{22k-11, 396\}$, where $k$ is the number of rows in the tree degree matrix. Furthermore, we give a necessary and sufficient condition when two tree degree sequences have edge disjoint caterpillar realizations.

\section{Preliminaries}
In this section, we give some formal definitions and lemmas we use throughout the paper. First, we formally define degree sequences and degree matrices, along with the different types of realizations we consider in this paper.
\begin{definition}
A degree sequence $D = d_1, d_2, \ldots d_n$ is a list of non-negative integers. A degree sequence is \emph{graphical} if there exist a simple graph $G$ whose degrees are exactly $D$. Such a graph is a \emph{realization} of $D$. A degree sequence $D$ is a \emph{tree degree sequence} if each degree is positive and $\sum_{i =1}^n d_i = 2n-2$. 

A degree $1$ vertex is called a \emph{leaf}. A degree sequence is a \emph{path degree sequence} if it has exactly two leaves. A realization of a tree degree sequence is called a \emph{caterpillar} if its non-leaf vertices form a path. This path of non-leaf vertices is called the \emph{backbone}.
\end{definition}

\begin{definition}
A matrix  $\mathbf{D}=\{\{d_{1,1},d_{1,2},\ldots d_{1,n} \},\{d_{2,1},d_{2,2},…d_{2,n} \},\ldots \{d_{k,1},d_{k,2},…d_{k,n}\}\}$ of non-negative integers is called a \emph{degree sequence matrix}. 

A degree sequence matrix is a \emph{tree degree sequence matrix} if each row is a tree degree sequence. A tree degree matrix has no common leaves if for each $i,j,l$, $d_{i,j} = 1 \implies d_{l,j} \ne 1$.

An edge colored simple graph $G$ is called a \emph{realization} of  a degree matrix $\mathbf{D} \in \mathbb{N}^{k\times n}$, if it is colored with $k$ colors, and for each color $c_i$, the subgraph containing the edges with color $c_i$ is a realization of the $i^{\mathrm{th}}$ row of $\mathbf{D}$. A realization is called \emph{caterpillar realization} if for each color, the corresponding subgraph is a caterpillar.

\end{definition}

The Erd{\H o}s-Gallai theorem gives necessary and sufficient conditions when a degree sequence is graphical.
\begin{theorem}{\bf  \cite{eg}} \label{theo:E-G} A degree sequence $f_1 \ge f_2 \ge \ldots \ge f_n$ is graphical if and only if the sum of the degrees is even, and for each $1 \le s \le n$ the inequality
\begin{equation}
\sum_{i = 1}^s f_i \le s(s-1) + \sum_{j=s+1}^n \min\{s, f_j\}\label{eq:E-G}
\end{equation}
holds.
\end{theorem}

We refer to the inequalities in Equation~\ref{eq:E-G} as Erd{\H o}s-Gallai inequalities, or E-G inequalities for short.

When a degree sequence is a sum of tree degree sequences, then only the first few Erd{\H o}s-Gallai inequalities must be checked, as the following lemma states.
\begin{lemma} {\bf \cite{ghm}}\label{lem:k-tree-EG}
Let $F = f_1 \ge f_2 \ge \ldots \ge f_n$ be the sum of $k$ arbitrary tree degree sequences. Then the Erd{\H o}s-Gallai inequalities in (\ref{eq:E-G})
holds for any $s \ge 2k$.
\end{lemma}

In this paper, we will need a stronger statement summarized in the following lemma.
\begin{lemma}\label{lem:tree+pathEG}
Let $D$ be a $2\times n$ tree degree matrix, in which the second row is a path degree sequence. If $n \ge 6$, then the E-G inequalities for the summed degree sequence $f_j := d_{1,j}+d_{2,j}$ hold for any $s \ge 2$.
\end{lemma}
\begin{proof}
Notice that the sum of a tree degree sequence is $2n-2$, the sum of the remaining $n-2$ degrees is at least $n-2$. Also, a path degree sequence does not have a degree larger than $2$. Therefore, when $s=2$, the left-hand side of the E-G inequality is bounded above by
$$
f_1 + f_2 \le 2n-2 -(n-2) +2 \times 2 = n + 4,
$$
The right-hand side is precisely
$$
2 + \sum_{j=3}^n \min\{2,f_j\} = 2 +2(n-2) = 2n-2,
$$
since each $f_j$ is at least $2$. Then it is sufficient to show that
$$
n+4 \le 2n-2
$$
which holds when $6\le n$.

When $s\ge 3$, we have on the left-hand side of the E-G inequality that
$$
\sum_{i=1}^s f_i \le 4n-4 - 3(n-s) +2 = n + 3s -2,
$$
since the total sum of the degrees is $4n-4$, and every column sum is at least $3$, except at most two of them. For similar reasons, we have the lower bound of the right-hand side of the E-G inequality:
$$
s(s-1) + 3(n-s) -2 \le s(s-1) + \sum_{j = s+1}^n  \min\{s, f_j\}.
$$
Therefore, the inequality holds as long as
$$
n + 3s -2 \le s(s-1) + 3(n-s) -2, 
$$
that is,
$$
0 \le s^2 - 7 s + 2n = s^2 - 7s +12 +x = (s-3)(s-4) +x,
$$
where $x \ge 0$. Since $s\ge 3$, we have that the inequality holds.
\end{proof}

In this paper, we are interested in the caterpillar realizations of tree degree matrices. Our main conjecture is the following:
\begin{conjecture}\label{con:caterpillar}
Let $D=\{\{d_{1,1},d_{1,2},\ldots d_{1,n} \},\{d_{2,1},d_{2,2},…d_{2,n} \},\ldots \{d_{k,1},d_{k,2},…d_{k,n}\}\}$ be a tree degree matrix without common leaves.
Then $D$ has a caterpillar realization.
\end{conjecture}

A special case is when the degree matrix contains path degree sequences without common leaves. It is well known that such degree matrices have caterpillar realizations, formally stated and proved in the following lemma:
\begin{lemma}\label{lem:allpaths}
Let $\mathbf{D}=\{\{d_{1,1},d_{1,2},\ldots d_{1,n} \},\{d_{2,1},d_{2,2},…d_{2,n} \},\ldots \{d_{k,1},d_{k,2},…d_{k,n}\}\}$ be a tree degree matrix without common leaves. Furthermore, assume each row is a path degree sequence. Then, $\mathbf{D}$ has edge disjoint path realizations.
\end{lemma}
\begin{proof}
We are going to explicitly construct these realizations. This construction is known as the Waleczki construction \cite{alspach2008}.

First observe that $n \ge 2k$, otherwise $D$ cannot accommodate the $2k$ leaves with at most one leaf in each column.
Without loss of generality (since we can rearrange the rows and columns), we can say that $d_{1,1}=1$, $d_{1,\left\lceil\frac{n+2}{2}\right\rceil}=1$, $d_{2,2}=1$, $d_{2,\left\lceil\frac{n+2}{2}\right\rceil+1}=1$, $\ldots$ $d_{k,k} = 1$, $d_{k,\left\lceil\frac{n+2}{2}\right\rceil+k-1}=1$. Then the  $i^{\mathrm{th}}$ path contains the edges $(v_i, v_{i+1})$, $(v_{i+1},v_{n+i-1})$, $(v_{n+i-1},v_{i+2})$, $(v_{i+2},v_{n+i-2})$, etc., where $n+i-j$ is considered modulo $n$, taking a value from the set $\{1,2,\ldots, n\}$.
\end{proof}

Some of our proofs are based on induction using the existence of rainbow matchings. We define them below.
\begin{definition}
Let $G$ be an edge-colored simple graph. A \emph{rainbow matching of size $k$} of $G$ is a matching of size $k$ in $G$ such that no two edges have the same color. 
\end{definition}

\section{Sufficient and necessary condition for two edge-disjoint caterpillar realizations}

B\'erczi \emph{et al.} \cite{bckm} gave the following example of a tree degree matrix:
$$\mathbf{D}= \left(
\begin{array}{ccccccccccc}
5 &2 &2 &2 &2 &2 &1 &1 &1 &1 &1 \\
5 &2 &2 &2 &2 &2 &1 &1 &1 &1 &1
\end{array}
\right)
$$
has edge disjoint tree realizations, but does not have edge disjoint caterpillar realizations. For $\mathbf{D}$ to have a caterpillar realization, each vertex can have at most two adjacent non-leaf edges per caterpillar. Notice that the first vertex has degree 10. At most $2\cdot 2$ of these can be non-leaf edges. So, this vertex is adjacent to at least 6 vertices which are leaves. However, there are only $5$ vertices which are leaves in any of the trees. As you can see, it is a naturally necessary condition that the maximum summed degree cannot be larger than 4 more than the number of vertices which are leaves in any of the trees. In this section, we show that this together with the condition that the summed degree sequence is graphical, are necessary and sufficient conditions to produce edge disjoint caterpillar realizations. 

\begin{theorem}
   Let $D$ be a $2\times n$ degree sequence matrix. Then $D$ has a caterpillar realization if and only if the following conditions hold:
  \begin{enumerate}
      \item\label{theo2cond1} For both $i=1$ and $i=2$, 
      $$
      \sum_{j=1}^n d_{i,j} = 2n -2.
      $$
      \item\label{theo2cond2} The degree sequences
      $$
      d_{1,1}+d_{2,1}, d_{1,2}+d_{2,2}, \ldots, d_{1,n}+d_{2,n}
      $$
      are graphical.
      \item\label{theo2cond3} It holds that
      $$
      d_{\max} \le |S| -4
      $$
      where
      $$
      d_{\max} := \max_{j}\left\{d_{1,j}+d_{2,j}\right\}
      $$
      and
      $$
      S := \left\{j\mid \min\{d_{1,j}, d_{2,j}\} = 1\right\}.
      $$
  \end{enumerate}
\end{theorem}
\begin{proof}
Conditions~\ref{theo2cond1}.~and~\ref{theo2cond2}.~are clearly necessary. Condition~\ref{theo2cond3}.~is also necessary, since any non-leaf vertex, will have at most two non-leaf neighbors in a caterpillar realization. If two caterpillar realizations are edge-disjoint, at least $d_{1,i}+d_{2,i}-4$ leaves must be a neighbor of $v_i$ in one of the caterpillar realizations.

Now we show that the conditions are also sufficient. Let $D$ be a $2\times n$ degree matrix that satisfies the conditions in the given theorem. Then the minimum column sum in $D$ is either $3$ or $2$. If the minimum sum is $3$, then there is a caterpillar realization, according to Theorem~\ref{theo:case_k2}. Observe that the non-trivial corollary that the necessary conditions holds if the minimum column sum is $3$. If the minimum column sum is $2$, then either there exist $j_1 \ne j_2$, such that $d_{1,j_1}>2$ and $d_{2,j_2} > 2$ or there does not exist two such distinct numbers $j_1$ and $j_2$.

Suppose such $j_1$ and $j_2$ exist. Order the columns in decreasing order by their column sums, and w.l.o.g. let $d_{1,1}>2$ (we can reorder the degree sequences if not). If $\exists j_1 \ne j_2 \in \{1,2\}$ such that $d_{1,j_1}>2$ and $d_{2,j_2}>2$, then fix such $j_1$ and $j_2$. Otherwise, let $j_1$ be $1$ and let $j_2$ be the smallest index for which $d_{2,j_2}>2$.

Let $D'$ denote the degree matrix we get from $D$ by removing a column with sum $2$ and subtracting $1$ both from $d_{1,j_1}$ and $d_{2,j_2}$. We are going to prove that $D'$ also satisfies the conditions given in the theorem.

Clearly, $D'$ is a tree degree matrix. Also, we remove a vertex that has a leaf, but also removed $1$ from the largest degree. If the first vertex has the unique largest summed degree in $D$, then it will still be largest in $G'$ (though may not be unique). Thus, condition~\ref{theo2cond3} from the theorem also holds for $D'$. If the first vertex does not have a unique largest summed degree, then the inequality in condition~\ref{theo2cond3} cannot be sharp for $D$, and thus will also hold for $D'$. Indeed, either $d_{1,1}+d_{1,2} \ge f_1$ or $d_{2,1}+d_{2,2} \ge f_1$, due to pigeonhole principle (it is possible that both degree sums are exactly $f_1$). Any tree with two vertex degrees $d_1$ and $d_2$ has at least $d_1+d_2-2$ leaves, thus we get that $|S| \ge f_1 -2$.

Therefore, we only have to prove that the column sums of $D'$, $f'_j := d'_{1,j}+d'_{2,j}$, form a graphical degree sequence. To prove it, it is sufficient to show that the first four E-G inequalities hold, according to Lemma~\ref{lem:k-tree-EG}. If $f_1 := d_{1,1}+d_{2,1}$ is the unique largest degree, then the first E-G inequality will also hold for $f'$. Indeed, both sides of the E-G inequality are decreased by $1$ (compared to the first E-G inequality for $f$). If $f_1 = f_2$ and $j_2 \notin \{1,2\}$ then 
$$
f_1 + f_2 \le 2n-2 -(n-2) +4 = n+4
$$
therefore, $f_1$ is at most $\frac{n}{2}+3$. We need that
$$
\left\lfloor\frac{n}{2}\right\rfloor+2\le n-2
$$
which holds if $n\ge 7$. If $n=6$, then the only possible tree degree matrix in which neither of the rows are path degree sequence and both $f_1$ and $f_2$ are $5$ is
$$
\left(
\begin{array}{cccccc}
3 & 3 & 1 & 1 & 1 & 1\\
2 & 2& 3 & 1 & 1 & 1
\end{array}
\right),
$$
however, in this case the column sums are not graphical. Similarly, there are not any $2\times 5$ tree degree matrices with the given condition whose column sum would be graphical. Finally, if $f_1 = f_2 = f_3$, then 
$$
f_1 + f_2 + f_3 \le 4n-4 -2(n-3) = 2n+2.
$$
That is, $f_1$ is at most $\frac{2n+2}{3}$. We need that
$$
\left\lfloor\frac{2n+2}{3}\right\rfloor\le n-2
$$
which holds if $n\ge 6$. If $n=5$, then the only possible tree degree matrix in which $f_1 = f_2 = f_3$ and both rows are not path degree sequences is
$$
\left(
\begin{array}{ccccc}
3 & 2 & 1 & 1  & 1\\
1 & 2& 3 & 1 & 1
\end{array}
\right),
$$
however, in this case the column sums are not graphical. Therefore, whenever the column sums of $D$ are graphical, the column sums of $D'$ satisfy the first E-G inequalities.

If $f_2 > f_3$, then $f'$ satisfies all E-G inequalities, since from $s =2$, both sides are decreased by $2$ (compared to the E-G inequalities for $f$).

Let $G'$ be a caterpillar realization of $D'$, by induction on the number of vertices, we can assume that such a realization exists. Then we can get a caterpillar realization of $D$ from $G'$ by adding a new vertex $v$ to $G'$ and connecting $v$ to $v_{j_1}$ with an edge of the first color and to $v_{j_2}$ with an edge of the second color.

If there does not exist distinct $j_i$ and $j_2$, where $d_{1, j_1} > 2$ and $d_{2, j_2} > 2$, then there are three cases:
\begin{enumerate}
\item Both degree sequences are paths.
\item Only one of the degree sequences is a path.
\item There is only one vertex, $v_1$, such that $d_{1,1} > 2$ and $d_{2,1}>2$.
\end{enumerate}

If both degree sequences are paths, then any tree realization is also a caterpillar realization. Kundu's theorem says there is a tree realization if the sum of the degree sequences is graphical \cite{kundu2trees}.

If one of the degree sequences is a path, then without loss of generality, the second degree sequence is a path. When $n \le 6$, there are $5$ possible tree degree matrices satisfying that the first row is not a path degree sequence, the second row is a path degree sequence, there is at least one column with sum $2$ and the column sums form a graphical degree sequence:
\begin{enumerate}
\item $\left(
\begin{array}{ccccc}
\underline{3} & 2 & 1 & 1 & 1 \\
1 & \underline{2} & 2 & 2 & 1
\end{array}\right)$

\item $\left(
\begin{array}{cccccc}
\underline{3} & 2 & 2 & 1 & 1 & 1 \\
1 & \underline{2} & 2& 2 & 2 & 1
\end{array}\right)$

\item $\left(
\begin{array}{cccccc}
\underline{3} & 2 & 2 & 1 & 1 & 1 \\
2 & \underline{2} & 1& 2 & 2 & 1
\end{array}\right)$

\item $\left(
\begin{array}{cccccc}
\underline{3} & 2 & 2 & 1 & 1 & 1 \\
2 & \underline{2} & 2& 2 & 1 & 1
\end{array}\right)$

\item $\left(
\begin{array}{cccccc}
\underline{4} & 2 & 1 & 1 & 1 & 1 \\
1 & \underline{2} & 2& 2 & 2 & 1
\end{array}\right)$

\end{enumerate}

In each case, we obtain a tree degree matrix $D'$ by subtracting $1$ from the underlined entries and removing a column with sum $2$. These $D'$ matrices have caterpillar realizations since either they are path degree sequences with graphical column sum or the minimum degree is $3$ (or both). In each case, the caterpillar realization $G'$ can be extended to a caterpillar realization of $D$ by adding one more vertex $v$ and connecting $v$ to the vertices where $1$ was subtracted from the degree using the appropriate color.

Now we consider the case when $n\ge 7$. We prove the theorem by induction on the number of vertices. Assume that the columns of $D$ are in decreasing order by their column sum, and amongst the same column sums, order the vertices lexicographically based on the two entries in the column. Since the second row is a path degree sequence, and there is a column with sum $2$, at least one of the entries $d_{2,1}$ and $d_{2,2}$ are $2$. If $d_{2,2} =1$, then let $D'$ be the degree matrix we obtain by removing $1$ from $d_{2,1}$ and $d_{1,2}$ and removing a column with sum $2$. Otherwise, let $D'$ be the degree matrix we obtain by removing $1$ from $d_{1,1}$ and $d_{2,2}$ and removing a column with sum $2$. We show that $\left\{f'_j\right\} :=d'_{1,j}+d'_{2,j}$ is graphical. Observe that the second degree sequence in $D'$ is a path, and the number of columns in $D'$ is at least $6$.  Therefore, it is sufficient to show that the first E-G inequality holds, due to Lemma~\ref{lem:tree+pathEG}.

If at least the first three columns have the same column sum in $D$, then the largest degree is at most $\frac{4n-4-3(n-3)+2}{3}= \frac{n+7}{3}$. We need that 
$$
\frac{n+7}{3} \le n-2,
$$
that is,
$$
6.5 \le n,
$$
which holds.
Then the first E-G inequality will also hold for $f'$. If $f_{1}$ is the unique largest degree, then $f'_{1}$ is one of the largest degree in $f'$. Since $f'_{1} = f_{1}-1$, and $1$ is subtracted from the right-hand side of the first E-G inequality. That the first E-G inequality holds for $f$ implies that it also holds for $f'$. Therefore, the column sums of $D'$ form a graphical degree sequence. By induction hypothesis, $D'$ has a caterpillar realization, $G'$. Then $D$ also has a caterpillar realization by extending $G'$ with a vertex and connecting it to the vertices where $1$ was subtracted from the degree using the appropriate color.

Finally, if there is only one vertex such that $d_{1,j} > 2$ and $d_{2,j} >2$, then this is the vertex with the largest summed degree. We prove the following two observations:
\begin{enumerate}
\item The number of columns with degree sum $2$ is at most $4$. Indeed, observe that the first tree has $d_{1,1}$ leaves while the second tree has $d_{2,1}$ leaves. Since the number of vertices which are leaves in at least one of the trees must be at least $d_{1,1}+d_{2,1}-4$, at most $4$ vertices might be leaves in both trees.
\item The number of columns with degree sum $4$ is at least the number of columns with degree sum $2$. This is the direct consequence that the summed degree sequence is graphical, therefore the E-G inequality holds with $s=1$. That is, the number of vertices above the first vertex is at least $d_{1,1}+d_{2,1}$. Also observe that the number of vertices with degree sum smaller than $4$ is $d_{1,1}+d_{2,1}$ minus the number of vertices with degree sum $2$. 
\end{enumerate}
Therefore, we have the following $4$ possible sub-cases:
\begin{enumerate}
\item
$\left(
\begin{array}{ccccccccccc}
\underline{d_{1,1}} & 2                   & \ldots & 2&  2 & \ldots &2& 1& \ldots & 1&1 \\
d_{2,1}                   & \underline{2} & \ldots & 2& 1&  \ldots &1& 2& \ldots & 2& 1
\end{array}\right)$

\item
$\left(
\begin{array}{ccccccccccccc}
\underline{d_{1,1}} & 2                   & \underline{2}&\ldots & 2&  2 & \ldots &2& 1& \ldots & 1&1 &1 \\
\underline{d_{2,1} }& \underline{2} & 2                  &\ldots & 2& 1&  \ldots &1& 2& \ldots & 2& 1&1
\end{array}\right)$

\item
$\left(
\begin{array}{ccccccccccccccc}
\underline{\underline{d_{1,1}}} & 2                     &2& \underline{2}&\ldots & 2&  2 & \ldots &2& 1& \ldots & 1&1 &1 &1\\
\underline{d_{2,1} }& \underline{2} & \underline{2} &2                 &\ldots & 2& 1&  \ldots &1& 2& \ldots & 2& 1&1&1
\end{array}\right)$

\item
$\left(
\begin{array}{ccccccccccccccccc}
\underline{\underline{d_{1,1}}} & 2                     &2& \underline{2}&\underline{2}&\ldots & 2&  2 & \ldots &2& 1& \ldots & 1&1 &1 &1&1\\
\underline{\underline{d_{2,1}}}& \underline{2} & \underline{2} &2      &2 &\ldots & 2& 1&  \ldots &1& 2& \ldots & 2& 1&1&1&1
\end{array}\right)$

\end{enumerate}
In each case, let $D'$ be a tree degree matrix we obtain by removing all columns with degree sum $2$, removing $1$ from each underlined degree and removing $2$ from each double underlined degree. The so-obtained matrices will be tree degree matrices without common leaves. Therefore, $D'$ has a caterpillar realization $G'$. W.l.o.g., we might assume that the vertices that have degree $1$ in one of the degree sequences after removing $1$ or $2$ are leaves adjacent to an end vertex of the backbone of the caterpillar. 
We can construct a caterpillar realization of $D$ by adding appropriate number of vertices to $G'$ and connect these to vertices where degree $1$ or $2$ were subtracted using the appropriate color. It is easy to see that in each case, we can add these edges without introducing parallel edges. Since we added leaves to backbone vertices or to leaves that were adjacent to end vertices of the backbone, the so-obtained edge disjoint tree realization will also be a caterpillar realization.

\end{proof}

\section{Proving Conjecture~\ref{con:caterpillar}.~for $k \le 4$}
In this section we are going to prove Conjecture~\ref{con:caterpillar}.~for all $k\le 4$.
The proofs are based on induction. The base cases are the cases when each degree sequence is a path degree sequence. Those degree matrices have edge disjoint path realizations, according to Lemma~\ref{lem:allpaths}. In the inductive step, we will find rainbow matchings in sufficiently long paths. The following two lemmas state that such paths exist.
\begin{lemma}
    Let $\mathbf{D} \in \mathbb{N}^{k\times n}$ be a tree degree matrix without common leaves. Then in any caterpillar realization of $\mathbf{D}$, each caterpillar has a path of length at least $2k-1$.
    
\label{Lemma.lemma2.1}
\end{lemma}

\begin{proof}
We will show this by contradiction. Assume there exists a degree sequence that does have a path of length $2k-1$. Then, it has at most $2k-3$ internal nodes and at least $n-2k+3$ leaves. Each of the other tree degree sequences have at least two leaves. So altogether, there are at least $n-2k+3+2(k-1)=n+1$ leaves. However, there are only $n$ vertices. So, there must exist one vertex with two leaves, producing a contradiction. 
\end{proof}
\begin{lemma}
    Let $\mathbf{D} = \{D_1, D_2, \ldots, D_k \} \in \mathbb{N}^{k\times n}$ be a tree degree matrix without common leaves.  If $n\geq 2k+2$ and $k\ge 4$, then within $(k-1)$ arbitrary caterpillars of any caterpillar realization of $\mathbf{D}$, there exists a path of length at least $2k+1$.
\label{Lemma.lemma2.2}
\end{lemma}
\begin{proof}
    Assume, to the contrary, that there does not exist a path of $2k+1$ edges within an arbitrary $(k-1)$ $D_i$'s. Then, there exists a set of $k-1$ tree degree sequences in $\mathbf{D}$ such that every $D_i$ does not have a path of length $2k+1$ edges. In other words, each of them must have at most $(2k+1)-2 = 2k-1$ internal nodes, and thus, must have at least $n-2k+1$ leaves. Hence, there are at least 
$$(k-1)(n-2k +1)+2=kn-n-2k^2+3k+1$$
leaves, which can be at most $n$. From this, we get that
$$
n\le \frac{2k^2-3k-1}{k-2}.
$$
However, since $n \geq 2k + 2$, we must have that
$$
\frac{2k^2-3k-1}{k-2} < 2k+2,
$$
implying that $k \leq 3$, a contradiction.

%The number of internal nodes for each $D_i$ is $n-l_i$, where $l_i$ is the number of leaves on the realization of $D_i$.\\
%    
%    So, $n-l_i < 2k$ for $(k-1)$ $D_i$'s. Thus, $l_i\geq 2k-1$. This means that for the $(k-1)$ $D_i$'s, we have at least $(k-1)(n-(2k-1))$ leaves. Also, our tree degree sequence that is not included contains at least two leaves. So, our $(k-1)$ $D_i$'s can have a most $(n-2)$ leaves.\\
%    
%    We compare our minimum amount of leaves, $(k-1)(n-(2k-1))$, with the maximum amount of leaves, $(n-2)$ leaves. The two value will be equal when $n=2k-1+\frac{2k+1}{k-2} \geq 2k+2$. Thus, starting from $n=2k+2$, the $(k-1)(n-(2k-1))$ will always be greater than $(n-2)$ as $k-1\geq3$. As a result, $(k-1)(n-(2k-1)) \geq n-2$ when $k\geq 4, n\geq 2k+2$, which raises the contradiction. 
\end{proof}

The following lemma is on the existence of a certain vertex in a tree degree matrix without common leaves.
\begin{lemma}\label{lem:v-vertex-existance}
Let $\mathbf{D} \in \mathbb{N}^{k\times n}$ be a tree degree matrix without common leaves. Assume that not all rows are path degree sequences. Then there exists a column with the following properties:
\begin{enumerate}
\item The sum of the column is $2k-1$.
\item The column contains a $1$ in a row which is not a path degree sequence.
\end{enumerate}
\end{lemma}
The proof is given in \cite{ghm}.

Now, we are ready to prove the conjecture for $k\le 4$.
With only one tree degree sequence, it is clear that we have a disjoint caterpillar realization. 
For $k = 2$, the conjecture was proved in \cite{bckm}, however, here is a simplified proof, the proof when $k=2$.
\begin{theorem}\label{theo:case_k2}
Let $\mathbf{D}$ be a $2\times n$ tree degree matrix without common leaves. Then $\mathbf{D}$ has a caterpillar realization.
\end{theorem}
\begin{proof}
The proof is constructive, using an induction on the number of vertices. If both sequences are path sequences, then they have edge disjoint realizations, according to Lemma~\ref{lem:allpaths}. Assume that at least one of the degree sequences is not a path; w.l.o.g., we can assume that the first degree sequence is not a path. According to Lemma~\ref{lem:v-vertex-existance}, there is a vertex $v$ which is a leaf in the first degree sequence, and has degree $2$ in the second degree sequence (the two rows in $\mathbf{D}$ might have to be swapped). Let $v_j$ be a vertex with degree at least $3$ in the non-path degree sequence. Then removing the column representing vertex $v$ and subtracting $1$ from $d_{1,j}$ yields a tree degree matrix $\mathbf{D}'$ without common leaves. By our induction hypothesis, it has a caterpillar realization. 

Let $G'$ be a realization of $\mathbf{D}'$. Its caterpillar realization of the second row of $\mathbf{D}'$ contains at least one edge in its backbone. At either end, there is $1$ edge connecting a endpoint to the backbone. Altogether, they form a path of at least $3$ edges. At most two of them can be incident to $v_j$. Consider an edge not incident to $v_j$; let it be $(u,w)$. We can construct a caterpillar realization of $\mathbf{D}$ from $G'$ in the following way. Add vertex $v$ to $G'$. Connect $v$ and $v_j$ with an edge of the first color. Remove edge $(u,w)$ and connect $v$ to both $u$ and $w$ with an edge of the second color. The subgraph of each color is a caterpillar realization of the appropriate row of $\mathbf{D}$, and they are edge-disjoint. Indeed, the caterpillar of the first color in $G'$ is extended with a leaf, and $v_j$ is not a leaf in this caterpillar. Vertex $v$ is a degree $2$ vertex in the second caterpillar, either inserted into the backbone or inserted between a leaf and the adjacent last vertex of the backbone. In both cases, the resulting tree is a caterpillar.
\end{proof}

The proof is very similar for three and four caterpillars. Just instead of a single edge $(u,w)$ avoiding vertex $v_j$, we have to find a rainbow matching avoiding $v_j$ in appropriate paths. Since we will use this technique multiple times throughout the paper, we introduce it in a separate lemma.

\begin{lemma}\label{lem:caterpillar-induction}
Let $\mathbf{D} \in \mathbb{N}^{k\times n}$ be a tree degree matrix without common leaves. Let $\mathbf{D}' \in \mathbb{N}^{k\times n-1}$ 
be a tree degree matrix without common leaves that we obtain from $\mathbf{D}$ by deleting a column containing all $2$'s except a $1$ in row $i$, and subtracting $1$ from an entry $d_{i,j} >2$. Let $G'$ be an arbitrary caterpillar realization of $\mathbf{D}'$. For the realized caterpillar of row $l$, let $P^l$ be the path containing the backbone of the caterpillar and two additional edges connecting arbitrary leaves to the end vertices of the backbone. If $\bigcup_{l\ne i}P^l$ contains a rainbow matching of size $k-1$ avoiding $v_j$, then $\mathbf{D}$ has a caterpillar realization. 
\end{lemma}
\begin{proof}
We are going to explicitly construct the caterpillar realization of $\mathbf{D}$ from $G'$. We add a vertex $v$ to $G'$. Vertex $v_j$ is connected to $v$ with an edge of color $i$. For each edge $(u,w)$ in the rainbow matching, the edge is removed and both $u$ and $w$ are connected to $v$ with an edge of the color of the removed edge.

We claim this is a caterpillar realization of $\mathbf{D}$. Indeed, for each color, we got a caterpillar realization of the appropriate row. In case of color $i$, the caterpillar in $\mathbf{D}'$ is extended with a leaf, and the leaf is connected to a backbone vertex. For all other colors $l$, a degree $2$ vertex is inserted into $P^l$. The so-obtained tree is a caterpillar. No parallel edges are introduced, because the new edges are formed from $v$ and vertices incident to edges in a rainbow matching which specifically avoids $v_j$.
\end{proof}

\begin{theorem}\label{theo:3-caterpillar}
Let $\mathbf{D} \in \mathbb{N}^{3 \times n}$ be a tree degree matrix without common leaves. Then $\mathbf{D}$ has a caterpillar realization.
\end{theorem}
\begin{proof}
The proof is again constructive, using an induction on the number of vertices. The base cases are the tree degree matrices in which each row is a path degree sequence.
In those cases, Lemma~\ref{lem:allpaths} provides edge disjoint path realizations.

Assume not all the rows are path degree sequences. According to Lemma~\ref{lem:v-vertex-existance}, there exists a column $l$ which contains, w.l.o.g., $1$ in the first row and $2$ in the other two rows. We can also assume that the first row is not a path degree sequence, implying there is a vertex $v_j$ such that $d_{1,j} \ge 3$. Consider $\mathbf{D}'$ obtained from $\mathbf{D}$ by removing column $l$ and subtracting $1$ from $d_{1,j}$. Matrix $\mathbf{D}'$ is a tree degree matrix without common leaves, and based on the inductive assumption, it has a caterpillar realization. Let the union of these caterpillars be denoted by $G'$. 

We want to find a rainbow matching in the remaining two rows avoiding $v_j$. The realized caterpillars of the second and third rows both contain a path of length at least $5$, according to Lemma~\ref{Lemma.lemma2.1}. In both paths, at most $2$ edges are incident to $v_j$, so there are at least $3$ edges in each caterpillar not incident to $v_j$. These $3$ edges form a path of length $3$ or a path of length $2$ with a separated edge. Suppose all three edges of these edges from one caterpillar are blocked by the other caterpillar. In both configurations, at most two of the three edges in one of the caterpillars can block all the $3$ edges in the other caterpillar, as shown in Figure~\ref{fig:3+3edges}. 
Therefore, there exists at least one of the three edges, call it $e_1$, not incident to $v_j$ and not adjacent to some other edge $e_2$ in the other caterpillar. Furthermore, $e_2$ is not incident to $v_j$. Therefore, $e_1$ and $e_2$ form a rainbow matching with two prescribed colors and avoid $v_j$. By Lemma~\ref{lem:caterpillar-induction}, $\mathbf{D}$ has a caterpillar realization.
\end{proof}
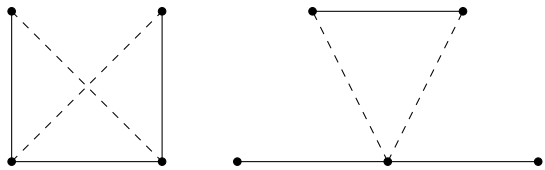
\begin{figure}
\setlength{\unitlength}{1cm}
\begin{center}
	\begin{tikzpicture}
\draw (0,2) -- (0,0) -- (2,0) -- (2,2);
\draw (3,0) -- (7,0);
\draw (4,2) -- (6,2);
\draw [dashed] (0,2) -- (2,0);
\draw [dashed] (0,0) -- (2,2);
\draw [dashed] (5,0) -- (4,2);
\draw [dashed] (5,0) -- (6,2);

\draw[black, fill= black] (0,0) circle(0.05);
\draw[black, fill= black] (2,0) circle(0.05);
\draw[black, fill= black] (0,2) circle(0.05);
\draw[black, fill= black] (2,2) circle(0.05);
\draw[black, fill= black] (3,0) circle(0.05);
\draw[black, fill= black] (5,0) circle(0.05);
\draw[black, fill= black] (7,0) circle(0.05);
\draw[black, fill= black] (4,2) circle(0.05);
\draw[black, fill= black] (6,2) circle(0.05);
\end{tikzpicture}
\end{center}
\caption{Only two dashed edges can both block three solid edges . See the text for details.}\label{fig:3+3edges}
\end{figure}

The proof for $k=4$ uses similar ideas, however, we need further base cases where $n \le 10$. Also, finding an appropriate rainbow matching is not easy. So, we separately present it in the following lemma.

\begin{lemma}
    Let $\mathbf{D}=\{D_1,D_2,D_3,D_4\} \in \mathbb{N}^{4\times n}$ be a tree degree matrix without common leaves. Let $G$ be a caterpillar realization of $\mathbf{D}$. Assume $v_j$ is an arbitrary vertex, and $G' \subset G$ is a caterpillar realization of an arbitrary three of the four tree degree sequences. If $n \geq 10$, there exists a rainbow matching of size three in ${G'\setminus \{v_j\}}$.
\label{lem:rainbow-k-4}
\end{lemma}
\begin{proof}

By applying Lemma \ref{Lemma.lemma2.2} and choosing $k=4$, we derive a special case for four tree degree sequences. Within any three out of four tree degree sequences, there exists a path of length at least 9. Let the degree sequence with the longest path be colored green, and the other two be colored blue and red. %Let $S$ be the rainbow matching set.
We have three cases. Case 1: $v_j$ is an internal node of the green degree sequence. Case 2: $v_j$ is a leaf of the green degree sequence. Case 3: The green path does not contain $v_j$. We will illustrate these three cases separately.

\begin{figure}[H]

\setlength{\unitlength}{1cm}
\begin{center}
\begin{tikzpicture}
\draw (-4,0) -- (-3,0) -- (-2, 0) -- (-1, 0);
\draw (1,0) -- (2,0) -- (3,0) -- (4,0) -- (5,0);

\draw [dashed] (-1, 0) -- (0,0) -- (1,0);

\draw [gray, ultra thick] (-5,-1) -- (-4,0);
\draw [gray, ultra thick] (-4,-1.4) -- (-4,0);

\draw [gray, ultra thin] (-5,1) -- (-4,0);
\draw [gray, ultra thin] (-4,1.4) -- (-4,0);

\draw [gray, ultra thick] (6,-1) -- (5,0);
\draw [gray, ultra thick] (5,-1.4) -- (5,0);

\draw [gray, ultra thin] (6,1) -- (5,0);
\draw [gray, ultra thin] (5,1.4) -- (5,0);

\draw[gray, fill= gray] (-5,-1) circle(0.07);
\draw[gray, fill= gray] (-4,-1.4) circle(0.07);
\draw[gray, fill= gray] (6,-1) circle(0.07);
\draw[gray, fill= gray] (5,-1.4) circle(0.07);

\draw[gray, fill= gray] (-5,1) circle(0.03);
\draw[gray, fill= gray] (-4,1.4) circle(0.03);
\draw[gray, fill= gray] (6,1) circle(0.03);
\draw[gray, fill= gray] (5,1.4) circle(0.03);

\draw[black, fill= black] (-4,0) circle(0.05);
\draw[black, fill= black] (-3,0) circle(0.05);
\draw[black, fill= black] (-2,0) circle(0.05);
\draw[black, fill= black] (-1,0) circle(0.05);
\draw[black, fill= black] (0,0) circle(0.05) node[anchor=south] {$v_j$};
\draw[black, fill= black] (1,0) circle(0.05);
\draw[black, fill= black] (2,0) circle(0.05);
\draw[black, fill= black] (3,0) circle(0.05);
\draw[black, fill= black] (4,0) circle(0.05);
\draw[black, fill= black] (5,0) circle(0.05);

\end{tikzpicture}
\end{center}

\caption{The graph shows the situation when $v_j$ is an internal node of the green path (longest path draw in black). The thick gray path represents the blue path, while the thin gray edges represent the red edges. The dotted edges represent the blocked edges.}
\label{fig:basicIllustration}

\end{figure}
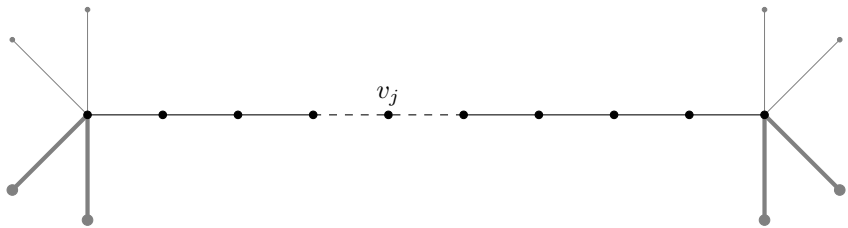

Fig. \ref{fig:basicIllustration} illustrates the first general case when $v_j$ is an internal node of the green degree sequence. It only includes the longest path in green degree sequence. As we are considering the graph $G'\setminus \{v_j\}$, the edges connected to $v_j$ are not considered. Two endpoints of the green path are leaves. So, they can not be leaves of the blue and red degree sequences. Hence, these endpoints must each be adjacent to two red and two blue edges. Call those eight edges endpoint edges. We have three possible scenarios.\\

Scenario 1: Less than two of the endpoint edges are incident to $v_j$.

At most one endpoint edge is blocked by $v_j$. Assume the color of this edge is blue. Consider the two red endpoint edges at this end of the green path, and the two blue endpoint edges at the opposite end of the green path. If none of these endpoint edges are incident to the two endpoints of the green path, then choose one of the red edges. It blocks at most one of the blue edges, so we have a pair of red and blue edges which are not adjacent. If one of these endpoint edges are incident to the two endpoints of the green path, then w.l.o.g. we can assume that it is a blue edge. Select the other blue endpoint edge, it blocks at most one of the red edges, therefore we again have a red and a blue edge that are not adjacent. For the green edges, we know that each blue and red edge in our rainbow matching set will block one leaf in green and at most another two edges in the green path. Also, $v_j$ blocks two green edges. Altogether, at most eight edges in the green path are blocked and there must exist one green edge that we can select. Therefore, we will find a rainbow matching of size three. \\

Scenario 2: Two endpoint edges of the same color are incident to $v_j$.

W.l.o.g., we can assume that the two endpoint edges incident to $v_j$ are blue. Select any of the red endpoint edges, call it $e$. It is adjacent to at most $4$ blue edges, one of these blue edges is also incident to $v_j$, and there is another blue edge incident to $v_j$. However, there are at least $7$ blue edges in the path of the blue caterpillar. So, there must be at least $2$ blue edges which are neither adjacent to $e$ nor incident to $v_j$, call them $f_1$ and $f_2$. The vertex $v_j$ blocks $2$ green edges from the green path. Edge $e$ blocks at most $3$ green edges from the green path. There are at least $4$ remaining green edges. It is impossible that both $f_1$ and $f_2$ blocks all these $4$ edges. Select a blue edge from $\{f_1,f_2\}$ that does not block the green edges incident to $e$ or $v_j$. Also, select $e$ and the green edge that is not adjacent to the selected blue edge, $e$, or $v_j$. These three edges form the appropriate rainbow matching.\\
%The general idea is exactly the same as situation 1. Assume the endpoint edges are blue. Observe that the other two blue endpoint edges cannot be the same otherwise there would be a cycle of length $3$, however, we know that the blue edges coming from a tree, which does not contain any cycle. So select any of the blue edges not incident to $v_j$. Then, we will have exactly the same setting as what we have in scenario 1. Following the same construction, we find a rainbow matching of size three.\\

Scenario 3: Two endpoint edges of different color are adjacent to $v_j$.

%From both ends of the green path, there must exist one color with only one edge that is not blocked by $v_j$. Consider those two edges. Select an edge that is not adjacent to the other end of out green path to be in $S$. Now we have scenario 1.
%
In this scenario, there are two blue edges and a red edge adjacent to one end of the green path. There are two red edges and a blue edge adjacent to the other end of the green path. None of these blue of red edges are incident to $v_j$. Even if one of these edges are the same (the two ends of the green path are connected with a red or a blue edge), there is an edge at one of the ends of the green path, w.l.o.g., we can say it is a blue edge, and there are two red edges at the other end. The blue edge can block at most one of the red edges. We have disjoint red and blue edges at the two ends of the green path. They block at most $6$ of the green edges, and $v_j$ blocks two of the green edges. So, there a green edge not adjacent to the selected red and blue edges and not incident to $v_j$. A pair of these disjoint red and blue edges, along with the green edge, form the appropriate rainbow matching.

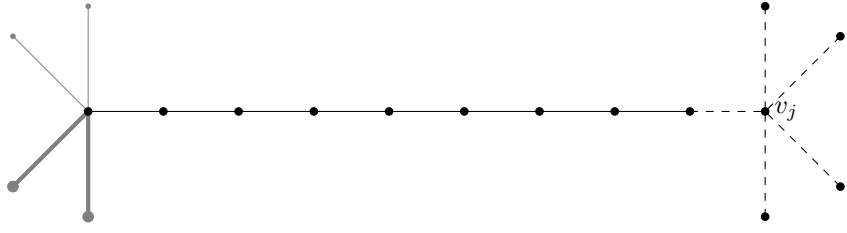
\begin{figure}[H]

\setlength{\unitlength}{1cm}
\begin{center}
\begin{tikzpicture}
\draw (-4,0) -- (-3,0) -- (-2, 0) -- (-1, 0) -- (0,0) -- (1,0) -- (2,0) -- (3,0) -- (4,0);

\draw [dashed] (4,0) -- (5,0);

\draw [gray, ultra thick] (-5,-1) -- (-4,0);
\draw [gray, ultra thick] (-4,-1.4) -- (-4,0);

\draw [gray, ultra thin] (-5,1) -- (-4,0);
\draw [gray, ultra thin] (-4,1.4) -- (-4,0);

\draw [dashed] (6,-1) -- (5,0);
\draw [dashed] (5,-1.4) -- (5,0);

\draw [dashed] (6,1) -- (5,0);
\draw [dashed] (5,1.4) -- (5,0);

\draw[gray, fill= gray] (-5,-1) circle(0.07);
\draw[gray, fill= gray] (-4,-1.4) circle(0.07);
\draw[black, fill= black] (6,-1) circle(0.05);
\draw[black, fill= black] (5,-1.4) circle(0.05);

\draw[gray, fill= gray] (-5,1) circle(0.03);
\draw[gray, fill= gray] (-4,1.4) circle(0.03);
\draw[black, fill= black] (6,1) circle(0.05);
\draw[black, fill= black] (5,1.4) circle(0.05);

\draw[black, fill= black] (-4,0) circle(0.05);
\draw[black, fill= black] (-3,0) circle(0.05);
\draw[black, fill= black] (-2,0) circle(0.05);
\draw[black, fill= black] (-1,0) circle(0.05);
\draw[black, fill= black] (0,0) circle(0.05);
\draw[black, fill= black] (1,0) circle(0.05);
\draw[black, fill= black] (2,0) circle(0.05);
\draw[black, fill= black] (3,0) circle(0.05);
\draw[black, fill= black] (4,0) circle(0.05);
\draw[black, fill= black] (5,0) circle(0.05) node[anchor=west] {$v_j$};

\end{tikzpicture}
\end{center}
\caption{The graph shows the situation when $v_j$ is a leaf of the green path (longest path draw in black). The thick gray path represents the blue path, while the thin gray edges represent the red edges. The dotted edges represent the blocked edges.}
\label{fig:case2Illustration}

\end{figure}
Fig. \ref{fig:case2Illustration} illustrates the second general case. In this case, the edges adjacent to one end of the green path are all blocked. One of the blocked edges is a green edge in the path, and at least four of the blocked edges are from the remaining two colors.

Consider the other end of the green path. At most one color has an edge that connects to the other end of the path. Assume that edge has color blue. If no edge connects to the other end of the path, choose one arbitrary edge as blue.  Select a blue edge that is not adjacent to the other end of the green path as our first edge for the rainbow matching set. Assume the other end of the blue edge is $v_i$. Then, we need to find a red edge that is not adjacent to either ends of the green path or $v_i$. By Lemma. \ref{Lemma.lemma2.1}, the red path must contain at least 7 edges. Each of the three vertices will block at most two red edges in the red path so there must exist one red edge left over. Select that red edge to be in the rainbow matching set. 

Now we find the green edge. The blue edge blocks one green leaf and another two green edges. The red edge will block four green edges. Also, $v_j$ blocks one green edge. Altogether, at most eight green edges are blocked. Since there are nine green edges, we can always select one green edge to put in the rainbow matching set. We constructed the appropriate rainbow matching set of size three.

Finally, in Case 3, $v_j$ is not on the green path. In that case, we can find disjoint red and blue edges not incident to $v_j$, see the proof of Theorem~\ref{theo:3-caterpillar}. These two edges block at most $8$ edges from the green path, so there is a green edge which is not adjacent to the selected red and blue edges and also not incident to $v_j$. 

\end{proof}

% I took out this theorem
\begin{comment}
\begin{theorem}
   Let $D=D_1,D_2,D_3,D_4$ be a realization of four tree degree sequence without common leaves. Assume $v_j$ is an arbitrary vertex and $G$ is the graph formed by an arbitrary three of the four tree degree sequences. If $n=9$, there exists a rainbow matching of size three in ${G/{v_j}}$.
\end{theorem}

\begin{proof}
If this is proved, we only need to prove the case n=8 and n=9.
\end{proof}
\end{comment}

We are now ready to prove Conjecture~\ref{con:caterpillar} for $k=4$.

\begin{theorem}
Let  $\mathbf{D} \in \mathbb{N}^{4\times n}$ be a tree degree matrix without common leaves. Then $\mathbf{D}$ has a caterpillar realization.
 
 \label{Theorem.theorem2.4}
    
\end{theorem}

\begin{proof}

The proof is constructive and based on induction. The base cases of the induction are those tree degree matrices that contain only path degree sequences and the matrices with at most $10$ vertices. If all rows are path degree sequences, then there exists a caterpillar realization by Lemma~\ref{lem:allpaths}. Up to permuting rows and columns, there are only $14$ tree degree matrices without common leaves. In the Appendix, we list these matrices and give a realization for each of them.

Now assume that $\mathbf{D} \in \mathbb{N}^{4 \times n}$ is a tree degree matrix without common leaves, where $n \ge 11$ and there is at least one row which is not a path degree sequence. Then there exists a column $l$ that contains a $1$ in a row not containing a path degree sequence, and all other entries in the column are $2$, according to Lemma~\ref{lem:v-vertex-existance}. Let $i$ be the row such that $d_{i,l} = 1$, and let $j$ be a column for which $d_{i,j} > 2$. Construct $\mathbf{D}'$ in the following way: remove column $l$, and subtract $1$ from $d_{i,j}$. Then $\mathbf{D}'$ is a tree degree matrix without common leaves, and based on the inductive assumption, it has a caterpillar realization. Let $G'$ be such a realization. According to Lemma~\ref{lem:rainbow-k-4}, the paths in the caterpillar realizations with color other than $i$ contain a rainbow matching avoiding $v_j$. By Lemma~\ref{lem:caterpillar-induction}, $\mathbf{D}$ has a caterpillar realization.

\end{proof}

\section{Degree Sequences on Large Amount of Vertices}

%For more than four tree degree sequences, our previous induction method no longer works. We do 
%not know that there exists a rainbow matching set of size $k-1$ within an arbitrary $k-1$ of the 
%trees.
For more than four tree degree sequences on a small number of vertices, it is hard to prove the existence of a rainbow matching of size $k-1$ within an arbitrary $k-1$ of the caterpillar realizations, while avoiding a prescribed vertex. It has been proved that edge disjoint tree realizations exist for any $\mathbf{D} \in \mathbb{N}^{k\times n}$ tree degree matrix without common leaves with $n\ge 4k-1$ if edge disjoint tree realizations exist for any $\mathbf{D} \in \mathbb{N}^{k\times (4k-2)}$ tree degree matrix without common leaves \cite{ghm}. We can prove a similar theorem with caterpillar realizations. For this, we need one more lemma on the lower bound of the length of the paths in caterpillar realizations.

\begin{lemma}\label{lem:inc-cat-paths}
Let $G$ be a caterpillar realization of $\mathbf{D} \in \mathbb{N}^{k\times n}$. Consider any $k-1$ of its caterpillars, and arrange them into increasing order based on the length of their paths containing their backbones, and the edges connecting leaves to the ends of the backbone. Then the $l^{th}$ longest path contains at least $(\frac{l-1}{l})n+2$ edges.
\end{lemma}
\begin{proof}
The proof is based on contradiction. Assume that $l^{th}$ longest path has at most $(\frac{l-1}{l})n+1$ edges. Then this caterpillar has at most $(\frac{l-1}{l})n$ internal nodes, and thus, at least $\frac{n}{l}$ leaves. Since the length of the paths are in increasing order, there are at least $l$ caterpillars with at least $\frac{n}{l}$ leaves. The other $k-l$ caterpillars have at least two leaves. Then there are at least
$$
l\times \frac{n}{l}+ (k-l)\times 2 = n + 2(k-l) > n
$$
leaves altogether, which produces a contradiction for $\mathbf{D}$ is a tree degree matrix without common leaves.
\end{proof}
\begin{theorem}
Let $k$ be an arbitrary positive integer. If any $\mathbf{D}' \in \mathbb{N}^{k\times (4k-2)}$ tree degree matrix without common leaves has a caterpillar realization, then any $\mathbf{D}\in \mathbb{N}^{k\times n}$ tree degree matrix without common leaves and $n\ge 4k-1$ has a caterpillar realization.
\end{theorem}
\begin{proof}
The proof is still based on induction. The base cases are the tree degree matrices in which each row is a path degree sequence and the tree degree matrices have dimension $k \times (4k-2)$. For any other tree degree matrix $\mathbf{D}$, we can construct the corresponding $\mathbf{D}'$ matrix (as we did in the proofs of the previous theorems), which has a realization $G'$. Next, we need to find a rainbow matching in the paths of $k-1$ selected caterpillars that avoids a prescribed vertex $v_j$. We claim we can find this rainbow matching in a greedy way. Arrange the caterpillars in increasing order based on the length of their longest path. We know that the $l^{\mathrm{th}}$ caterpillar has a path of length at least $(\frac{l-1}{l})n+2$, according to Lemma~\ref{lem:inc-cat-paths}. Therefore, it has a matching of size at least $\left\lceil \frac{(\frac{l-1}{l})n+2}{2}\right\rceil = \left\lceil \frac{(l-1)n}{2l}\right\rceil+1$. We know that $n$ is at least $4k-2$ and $k$ is at least $l+1$, thus the matching has a size at least
$$
\left\lceil \frac{2(l-1)(2l+1)}{2l}\right\rceil +1 \ge 2(l-1)+2.
$$
We already selected $l-1$ edges in the rainbow matching that block at most $2(l-1)$ edges in the matching of the $l^{\mathrm{th}}$ color. Vertex $v_j$ can block at most one edge. Therefore, we have an edge of the $l^{\mathrm{th}}$ color that is not adjacent to any of the so-far selected edges, and is not incident to vertex $v_j$. We can select this edge for the rainbow matching.

Since we are able to find a rainbow matching avoiding $v_j$, $\mathbf{D}$ has a caterpillar realization, according to Lemma~\ref{lem:caterpillar-induction}.

\end{proof}

We also introduce a theorem that unconditionally claims the existence of caterpillar realizations with large number of vertices.  
For this theorem, we have a new strategy to directly construct the caterpillar realization for $\mathbf{D}$. Treat a caterpillar as the union of leaves and a backbone. We define \emph{leg} as the edges that incident to a leaf. We also define \emph{backbone edges} as all the other edges. The construction strategy is to first construct all the legs and $3$ backbones and then the remaining backbones. The key point is that we will be able to find backbones as Hamiltonian paths in appropriate subgraphs that we obtain by removing the so-far used edges from the complete graph on the given backbone vertices. In general, we will denote this subgraph as $F$. 
The existence of these Hamiltonian paths is proved by a theorem similar to Ore's theorem, which states that for a finite and simple graph $G$, if $d_i + d_j \geq n$  for every pair of distinct non-adjacent vertices $i$ and $j$ of $G$, then $G$ must contain a Hamiltonian cycle \cite{ore1960,palmer1997}.

In our case, not all vertices satisfy the conditions of Ore's theorem. However, we will still be able to find Hamiltonian paths in a given graph $F$. Our strategy is based on the following observations:
\begin{enumerate}
\item All backbones are sufficiently long except the shortest three backbones. Actually, only the shortest backbone might have $o(n)$ length, where $n$ is the number of vertices; however, we can easily construct the three shortest backbones. This allows for a better lower bound on the number of vertices necessary to construct the remaining caterpillars.
\item In any tree degree matrix without common leaves, there is at most one vertex (that is, column) with total degree at least $\frac{2n}{3}$ if $n$ is sufficiently large. Actually, for any $c>\frac{1}{2}$, there is at most one vertex whose degree is at least $cn$ if $n$ is sufficiently large.
\item There are at most $11$ vertices whose degrees are larger than $\frac{n}{6}$ if $n$ is sufficiently large. Actually, for any $c'>0$, there are at most a constant number of vertices whose degrees are larger than $c'n$ if $n$ is sufficiently large.
\end{enumerate}
We are going to precisely state and prove these statements below. These observations provides us the following construction strategy.
\begin{enumerate}
\item We first construct the legs of the caterpillars and three shortest backbones. Then we construct all other backbones.
\item To construct the other backbones, we ``cap" all the backbone vertices whose degree is small in $F$, with vertices whose degree is large in $F$.  The vertices with small degree in $F$ are the same vertices with large degree in $\mathbf{D}$. There are constant number of these vertices. Furthermore, at most one of them might have degree larger than $\frac{2n}{3}$ in $\mathbf{D}$. Therefore, at most one of them might have too small of a degree in $F$. We cap this vertex for each backbone at the end of the first phase. Since all other small degree vertices have degree at least $\frac{n}{6}$ in $F$, we can easily find high degree vertices to cap them.
\item We fix the edges used for capping the small degree vertices in $F$, and extend them to a Hamiltonian path. The algorithm to find such a Hamiltonian path is very similar to Palmer's algorithm \cite{palmer1997} to find a Hamiltonian cycle in a graph satisfying the degree conditions in the hypothesis of Ore's theorem \cite{ore1960}.
\end{enumerate} 

Below we state and prove the lemmas concerning the degree properties. The first is a simple observation.
\begin{lemma}\label{lem:very-large-degree}
Let $\mathbf{D} \in \mathbb{N}^{k\times n}$ be a tree degree matrix without common leaves. Assume that $n \ge 6k-5$. Then there exists at most one vertex whose degree is larger than or equal to $\frac{2n}{3}$.
\end{lemma}
\begin{proof}
Assume to the contrary there exists at least $2$ vertices with degree at least $\frac{2n}{3}$. Then the total number of degrees is at least
$$
\frac{4n}{3} + (n-2)(2k-1) = 2kn +\frac{n}{3}- 4k +2.
$$
However, the total number of degrees is $k(2n-2)$. If
$$
2kn +\frac{n}{3}-4k  +2 \le k(2n-2),
$$
then
$$
n  \le 6k -6,
$$
a contradiction.
\end{proof}
The number of relatively high degree vertices is also small.
\begin{lemma}
Let $\mathbf{D} \in\mathbb{N}^{k\times n}$ be a tree degree matrix without common leaves. If $n\ge 22k-11$, then there are at most $11$ vertices with degree at least $\frac{n}{6}$.
\end{lemma}
\begin{proof}
Assume contrary. Then the sum of the degrees is at least
$$
12\cdot\frac{n}{6} + (n-12)(2k-1) = 2nk +n -24k +12.
$$
The total degree is $k(2n-2)$. Then it holds that
$$
2nk +n -24k + 12 \le k(2n-2)
$$
from which
$$
n \le 22k - 12,
$$
a contradiction.
\end{proof}

We are now ready to state and prove the following theorem.
\begin{theorem}
Let $\mathbf{D}\in \mathbb{N}^{n\times k}$ be a tree degree sequence without common leaves. Assume that $k\ge 5$ and  $n \ge \max\{22k-11, 396\}$. Then  $\mathbf{D}$ has a caterpillar realization. 
\end{theorem}
\begin{proof}
We explicitly construct a realization in two phases. In the first phase, we construct the legs of the caterpillars, the backbones of the $3$ shortest backbones and all the remaining edges of the largest degree vertex. In the second phase, we  construct the remaining backbones.

\noindent \emph{Phase I}

Let  $\mathbf{D} = \mathbf{D}_0, \mathbf{D}_1, \ldots, \mathbf{D}_m$ be a series of tree degree matrices, such that $\mathbf{D}_m$ contains only path degree sequences, and for any $l = 0, \ldots m-1$, $\mathbf{D}_{l+1}$ is obtained from $\mathbf{D}_l$ by removing a column containing all $2$'s except in row $i$, where the entry is $1$, and then subtracting $1$ from a $d_{i,j} > 2$.  According to Lemma~\ref{lem:v-vertex-existance}, we can always find a column with column sum $2k-1$ and entry $d_{i,j}$. Matrix $\mathbf{D}_m$ has edge disjoint path realizations, according to Lemma~\ref{lem:allpaths}.

Let $v$ denote the vertex with the largest column sum in $\mathcal{D}$. Furthermore, let the vertices adjacent to leaves in the paths be called ``end vertices".
Let $G$ be a subset of the above-mentioned edge disjoint path realizations of $\mathcal{D}_m$ containing the three paths corresponding to the caterpillars with shortest backbones in $\mathbf{D}$, the legs of the other paths and all edges incident to $v$. Observe that for each color, $v$ has at most two backbone edges. When it is incident to exactly two backbone edges, then at most one of these edges are incident to an end vertex. Furthermore, when $v$ has one backbone edge, that is, when $v$ is an end vertex, then this edge is not incident to another end vertex. We call the backbone edges of $v$ as ``capping edges".

Then going from $\mathbf{D}_m$ to $\mathbf{D}$, we construct the three before-mentioned caterpillars and all the legs of other caterpillars, adding one vertex $v$ to $G$ in each step. For each $\mathbf{D}_l$ to $\mathbf{D}_{l-1}$ transition, if the removed column contains a $1$ in a row not corresponding to the three caterpillars, then add a leg between $v$ and $v_j$ with color $i$ and find a rainbow matching avoiding $v_j$ in the backbones of the $3$ caterpillars and extend these caterpillars by pulling these edges onto $v$. Such rainbow matching exists, according to Lemma~\ref{lem:rainbow-k-4}. If row $i$ contains one of the caterpillars constructed in this phase, then connect $v_j$ to $v$ with color $i$, find a rainbow matching in the backbones of the other caterpillars, and pull them onto $v$. Such rainbow matching exists, according to Lemma~\ref{lem:rainbow-k-4}.

In this way, we construct all legs, the three caterpillars with the shortest backbones and all edges incident to $v$. Furthermore, we put all of these (appropriately colored) edges in $G$. For other backbones, the remaining degree of each vertex (that is, the difference of its degree with color $i$ in $G$ and the corresponding entry in row $i$ of $\mathbf{D}$) is either $1$ or $2$ (except for $v$). The vertices with remaining degree $1$ are exactly the end vertices of the backbones, as their all other degrees are used for legs (and $v$ is an end vertex if it has one capping edge). We are ready to enter Phase II.

\noindent \emph{Phase II}

We construct the backbones in increasing order according to their length. We add these backbones to $G$, and before each backbone, let $F$ denote the complement of $G$ restricted to the current backbone vertices. Since the legs of the current caterpillar are added, the two end vertices of the backbone are prescribed. Our task is to find a Hamiltonian path between these two vertices in $F$. Although the majority of the degrees is large in $F$, we cannot directly apply Ore's theorem, because there might exist a few small degree vertices. We are going to cap the small degree vertices with high degree vertices, and then extend them into a Hamiltonian path.

Let $m$ denote the size of $F$. From Lemma~\ref{lem:inc-cat-paths}, we know that $m$ is larger than $\frac{3n}{4}$. We know only $v$ can have a degree reaching $\frac{2n}{3}$, and at most $11$ vertices $G$ can have degrees reaching $\frac{n}{6}$. If a degree in $G$ is less than $\frac{2n}{3}$, then its degree in $F$ is at least
$$
\frac{3n}{4} -\frac{2n}{3} = \frac{n}{12}.
$$

All other vertices have degree less than $\frac{n}{6}$. Thus, their degree in $F$ is at least 
$$
\frac{3n}{4}-\frac{n}{6} = \frac{7n}{12}.
$$
Therefore, the sum of any two of these high degrees is at least $\frac{7n}{6}$.

If $v$ is a backbone vertex in the current caterpillar,  put its $1$ or $2$ capping edges into the set $E$. Then for each vertex $w$ in the backbone that has degree at least $\frac{n}{6}$ (but less than $\frac{2n}{3}$) in $G$, we distinguish $4$ cases:
\begin{enumerate}
\item Vertex $w$ is incident to a capping edge of $v$, and it is an end vertex. In this case, we do not have to find further capping edges of $w$.
\item Vertex $w$ is incident to a capping edge of $v$, and it is not an end vertex. Then we will find one more capping edge of $w$.
\item Vertex $w$ is not incident to a capping edge of $v$, and it is an end vertex. Then we will find one capping edge of  $w$.
\item Vertex $w$ is not incident to a capping edge of $v$, and it is not an end vertex. Then we will find two capping edges of $w$.
\end{enumerate}

We claim that we can find the necessary one or two neighbor vertices of $w$, denoted by $u_1$ and $u_2$ in $F$ that have high degree in $F$ (at least $\frac{7n}{12}$), not incident to any edge in $E$, and where at most one of them is an endpoint of the backbone.  Such neighbors must exist, because there are at most $30$ forbidden points (the at most $10$ other low degree vertices and for each of them, at most $2$ neighbors incident to their capping edges). However, these low degree vertices have degree at least $\frac{n}{12}$ in $F$, and $n$ is at least $396$. Therefore, there are at least $3$ neighbors which are not forbidden. We can select two of the three such that at most one of them is an endpoint of the backbone. We add edges $(w,u_1)$ and $(w,u_2)$ to $E$.

Now we construct the backbone. Arrange the backbone vertices in a cycle, starting and ending with the endpoints of the backbone such that vertices incident to the same edge in $E$ are neighbors. We set up such a permutation, since the endpoints do not have a common neighbor in edge set $E$. Then we apply an algorithm similar to Palmer's algorithm to construct a Hamiltonian path \cite{palmer1997}. While there remains two neighbor vertices $u_1$ and $u_2$ around the cycle in a clockwise direction not having an edge in $F$, we find a vertex pair $w_1$ and $w_2$ such that they are neighbors around the cycle, not both of them are endpoints of the backbone, and $(u_1,w_1) \in E$, $(u_2, w_2)\in E$, but $(w_1,w_2) \notin E$. 

By pigeonhole principle, such pair of vertices exists. Both $u_1$ and $u_2$ have high degree. The sum of their degrees is at least $\frac{7n}{6}$. So there must exist at least $\frac{n}{6}$ pairs of neighbor vertices such that $(u_1,w_1)\in E$ and $(u_2,w_2)\in E$. There are at most $23$ forbidden pairs from the at most $22$ pairs forming the edges in $E$ and the pair of endpoint vertices. However, $\frac{n}{6} > 23$.

We swap the appropriate arc of the cycle to make $u_1$ be a neighbor of $w_1$, and $u_2$ be a neighbor of $w_2$. With this operation, we decrease the amount of neighbor pairs $u_1,u_2$ around the cycle that do not have an edge between them in $F$. After applying this operation a finite number of times, the number of such neighbors will reach $0$. That is, there is a Hamiltonian path in $F$ between the two endpoint vertices of the backbone. 

Since for each degree sequence we can find a Hamiltonian path in $F$ between the two endpoints of the backbone, we can construct a caterpillar realization of $\mathbf{D}$.
\end{proof}

\section{Discussion}

In this paper, we considered the caterpillar realizations of tree degree matrices. We presented necessary and sufficient conditions when a $2\times n$ tree degree matrix has an edge disjoint caterpillar realization. Starting from $k=3$, it seems extremely hard to find necessary and sufficient condition for a caterpillar realization of a $k\times n$ tree degree matrix. However, the vertices having no common leaves seems to be a sufficient condition, that is, each vertex has degree $1$ in at most one of the degree sequences. We were able to prove that this condition is sufficient when $k \le 4$, or when $n\ge \max\{22k-11, 396\}$. Naturally, $n$ should be at least $2k$, and we also proved that the conjecture is true if it is true for any $n\le 4k-2$. However, it seems difficult to close the gap between $n=2k$ and $n= 4k-2$, though it is well known that the conjecture is true for $n = 2k$ \cite{alspach2008}. 

Since any caterpillar is a tree, our conjecture is also a conjecture for edge disjoint tree realizations. The ``no common leaves" condition forces the column sums to be more-or-less evenly distributed, that is, most of the column sums are $o(n)$. It is an open question if other conditions forcing evenly distributed column sums are sufficient for caterpillar (or edge disjoint tree) realizations. It also an open question of how many common leaves are necessary to find a counterexample of a tree degree matrix that has no caterpillar realizations.

\section*{Appendix}

Up to permutations of degree sequences and vertices, there are $14$ tree degree matrices on at most  $10$ vertices without common leaves. This appendix gives an example caterpillar realization for all of them.

If the number of vertices is $8$, there is only one possible tree degree matrix, each degree sequence is a path degree sequence (case 1).

If the number of vertices is $9$, there are $2$ possible cases: either all degree sequences are path degree sequences (case 2) or there is a degree $3$ (case 3).

If the number of vertices is $10$, there are $11$ possible cases: all degree sequences are path degree sequences (case 4), there is a degree 3 which might be on a vertex with a leaf (case 5) or without a leaf (case 6), there is a degree $4$ (case 7) or there are $2$ degree $3$'s in the degree sequences (cases 8-14).

The two $3$'s might be in the same degree sequence. The leaves on these two vertices might be in the same degree sequence (case 8) or in different degree sequences (case 9).

If the two degree $3$s are in different degree sequences, they might be on the same vertex (case 10) or on different vertices.

If the two degree $3$s are in different sequences, $D_i$ and $D_j$, and on different vertices $u$ and $v$, consider the degrees of $u$ and $v$ in $D_i$ and $D_j$ which are not $3$. They might be both $1$ (case 11), or maybe one of them is $1$ and the other is $2$ (case 12), or both of them are $2$. In this latter case, the degree $1$'s on $u$ and $v$ might be in the same degree sequence (case 13) or in different degree sequences (case 14).

The realizations are represented with an adjacency matrix, in which $0$ denotes the absence of edges, and for each degree sequence $D_i$, $i$ denotes the edges in the realization of $D_i$.

\begin{enumerate}
\item %1
$$\mathbf{D}= \left(
\begin{array}{cccccccc}
1 &2 &2 &2 &1 &2 &2 &2 \\
2 &1 &2 &2 &2 &1 &2 &2 \\
2 &2 &1 &2 &2 &2 &1 &2 \\
2 &2 &2 &1 &2 &2 &2 &1 
\end{array}
\right)
$$

$$
A= \left(
\begin{array}{cccccccc}
0 &1 &2 &2 &3 &3 &4 &4 \\
1 &0 &2 &3 &3 &4 &4 &1 \\
2 &2 &0 &3 &4 &4 &1 &1 \\
2 &3 &3 &0 &4 &1 &1 &2 \\
3 &3 &4 &4 &0 &1 &2 &2 \\
3 &4 &4 &1 &1 &0 &2 &3 \\
4 &4 &1 &1 &2 &2 &0 &3 \\
4 &1 &1 &2 &2 &3 &3 &0 
\end{array}
\right)
$$

\item %2
$$\mathbf{D}= \left(
\begin{array}{ccccccccc}
1 &2 &2 &2 &2 &1 &2 &2 &2 \\
2 &1 &2 &2 &2 &2 &1 &2 &2 \\
2 &2 &1 &2 &2 &2 &2 &1 &2 \\
2 &2 &2 &1 &2 &2 &2 &2 &1 
\end{array}
\right)
$$

$$
A= \left(
\begin{array}{ccccccccc}
0 &1 &2 &2 &3 &3 &4 &4 &0 \\
1 &0 &2 &3 &3 &4 &4 &0 &1 \\
2 &2 &0 &3 &4 &4 &0 &1 &1 \\
2 &3 &3 &0 &4 &0 &1 &1 &2 \\
3 &3 &4 &4 &0 &1 &1 &2 &2 \\
3 &4 &4 &0 &1 &0 &2 &2 &3 \\
4 &4 &0 &1 &1 &2 &0 &3 &3 \\
4 &0 &1 &1 &2 &2 &3 &0 &4 \\
0 &1 &1 &2 &2 &3 &3 &4 &0 
\end{array}
\right)
$$

\item %3
$$\mathbf{D}= \left(
\begin{array}{ccccccccc}
1 &3 &2 &2 &1 &2 &2 &2 &1 \\
2 &1 &2 &2 &2 &1 &2 &2 &2 \\
2 &2 &1 &2 &2 &2 &1 &2 &2 \\
2 &2 &2 &1 &2 &2 &2 &1 &2 
\end{array}
\right)
$$

$$
A= \left(
\begin{array}{ccccccccc}
0 &1 &0 &2 &3 &3 &4 &4 &2 \\
1 &0 &2 &3 &3 &4 &4 &1 &1 \\
0 &2 &0 &3 &4 &4 &1 &1 &2 \\
2 &3 &3 &0 &0 &1 &1 &2 &4 \\
3 &3 &4 &0 &0 &1 &2 &2 &4 \\
3 &4 &4 &1 &1 &0 &2 &0 &3 \\
4 &4 &1 &1 &2 &2 &0 &3 &0 \\
4 &1 &1 &2 &2 &0 &3 &0 &3 \\
2 &1 &2 &4 &4 &3 &0 &3 &0 
\end{array}
\right)
$$

\item %4
$$\mathbf{D}= \left(
\begin{array}{cccccccccc}
1 &2 &2 &2 &2 &1 &2 &2 &2 &2 \\
2 &1 &2 &2 &2 &2 &1 &2 &2 &2 \\
2 &2 &1 &2 &2 &2 &2 &1 &2 &2 \\
2 &2 &2 &1 &2 &2 &2 &2 &1 &2 
\end{array}
\right)
$$

$$
A= \left(
\begin{array}{cccccccccc}
0 &1 &2 &2 &3 &3 &4 &4 &0 &0 \\
1 &0 &2 &3 &3 &4 &4 &0 &0 &1 \\
2 &2 &0 &3 &4 &4 &0 &0 &1 &1 \\
2 &3 &3 &0 &4 &0 &0 &1 &1 &2 \\
3 &3 &4 &4 &0 &0 &1 &1 &2 &2 \\
3 &4 &4 &0 &0 &0 &1 &2 &2 &3 \\
4 &4 &0 &0 &1 &1 &0 &2 &3 &3 \\
4 &0 &0 &1 &1 &2 &2 &0 &3 &4 \\
0 &0 &1 &1 &2 &2 &3 &3 &0 &4 \\
0 &1 &1 &2 &2 &3 &3 &4 &4 &0 
\end{array}
\right)
$$

\item %5
$$\mathbf{D}= \left(
\begin{array}{cccccccccc}
1 &3 &2 &2 &2 &1 &2 &2 &2 &1 \\
2 &1 &2 &2 &2 &2 &1 &2 &2 &2 \\
2 &2 &1 &2 &2 &2 &2 &1 &2 &2 \\
2 &2 &2 &1 &2 &2 &2 &2 &1 &2 
\end{array}
\right)
$$

$$
A= \left(
\begin{array}{cccccccccc}
0 &1 &0 &2 &3 &3 &4 &4 &0 &2 \\
1 &0 &2 &3 &3 &4 &4 &0 &1 &1 \\
0 &2 &0 &3 &4 &4 &0 &1 &1 &2 \\
2 &3 &3 &0 &0 &0 &1 &1 &2 &4 \\
3 &3 &4 &0 &0 &1 &1 &2 &2 &4 \\
3 &4 &4 &0 &1 &0 &2 &2 &3 &0 \\
4 &4 &0 &1 &1 &2 &0 &0 &3 &3 \\
4 &0 &1 &1 &2 &2 &0 &0 &4 &3 \\
0 &1 &1 &2 &2 &3 &3 &4 &0 &0 \\
2 &1 &2 &4 &4 &0 &3 &3 &0 &0 
\end{array}
\right)
$$

\item %6
$$\mathbf{D}= \left(
\begin{array}{cccccccccc}
1 &2 &2 &2 &3 &1 &2 &2 &2 &1 \\
2 &1 &2 &2 &2 &2 &1 &2 &2 &2 \\
2 &2 &1 &2 &2 &2 &2 &1 &2 &2 \\
2 &2 &2 &1 &2 &2 &2 &2 &1 &2 
\end{array}
\right)
$$

$$
A= \left(
\begin{array}{cccccccccc}
0 &1 &0 &2 &3 &3 &4 &4 &0 &2 \\
1 &0 &2 &0 &3 &4 &4 &0 &1 &3 \\
0 &2 &0 &3 &4 &4 &0 &1 &1 &2 \\
2 &0 &3 &0 &4 &0 &1 &1 &2 &3 \\
3 &3 &4 &4 &0 &1 &1 &2 &2 &1 \\
3 &4 &4 &0 &1 &0 &2 &2 &3 &0 \\
4 &4 &0 &1 &1 &2 &0 &3 &3 &0 \\
4 &0 &1 &1 &2 &2 &3 &0 &0 &4 \\
0 &1 &1 &2 &2 &3 &3 &0 &0 &4 \\
2 &3 &2 &3 &1 &0 &0 &4 &4 &0 
\end{array}
\right)
$$

\item %7
$$\mathbf{D}= \left(
\begin{array}{cccccccccc}
1 &4 &2 &2 &1 &2 &2 &2 &1 &1 \\
2 &1 &2 &2 &2 &1 &2 &2 &2 &2 \\
2 &2 &1 &2 &2 &2 &1 &2 &2 &2 \\
2 &2 &2 &1 &2 &2 &2 &1 &2 &2 
\end{array}
\right)
$$

$$
A= \left(
\begin{array}{cccccccccc}
0 &1 &0 &0 &3 &3 &4 &4 &2 &2 \\
1 &0 &2 &3 &3 &4 &4 &1 &1 &1 \\
0 &2 &0 &3 &0 &4 &1 &1 &2 &4 \\
0 &3 &3 &0 &0 &1 &1 &2 &4 &2 \\
3 &3 &0 &0 &0 &1 &2 &2 &4 &4 \\
3 &4 &4 &1 &1 &0 &2 &0 &3 &0 \\
4 &4 &1 &1 &2 &2 &0 &0 &0 &3 \\
4 &1 &1 &2 &2 &0 &0 &0 &3 &3 \\
2 &1 &2 &4 &4 &3 &0 &3 &0 &0 \\
2 &1 &4 &2 &4 &0 &3 &3 &0 &0 
\end{array}
\right)
$$

\item %8
$$\mathbf{D}= \left(
\begin{array}{cccccccccc}
1 &3 &2 &2 &1 &3 &2 &2 &1 &1 \\
2 &1 &2 &2 &2 &1 &2 &2 &2 &2 \\
2 &2 &1 &2 &2 &2 &1 &2 &2 &2 \\
2 &2 &2 &1 &2 &2 &2 &1 &2 &2 
\end{array}
\right)
$$

$$
A= \left(
\begin{array}{cccccccccc}
0 &1 &0 &2 &0 &3 &4 &4 &2 &3 \\
1 &0 &0 &3 &3 &4 &4 &1 &1 &2 \\
0 &0 &0 &3 &4 &4 &1 &1 &2 &2 \\
2 &3 &3 &0 &0 &1 &1 &2 &0 &4 \\
0 &3 &4 &0 &0 &1 &2 &2 &4 &3 \\
3 &4 &4 &1 &1 &0 &2 &0 &3 &1 \\
4 &4 &1 &1 &2 &2 &0 &3 &0 &0 \\
4 &1 &1 &2 &2 &0 &3 &0 &3 &0 \\
2 &1 &2 &0 &4 &3 &0 &3 &0 &4 \\
3 &2 &2 &4 &3 &1 &0 &0 &4 &0 
\end{array}
\right)
$$

\item %9
$$\mathbf{D}= \left(
\begin{array}{cccccccccc}
1 &3 &3 &2 &1 &2 &2 &2 &1 &1 \\
2 &1 &2 &2 &2 &1 &2 &2 &2 &2 \\
2 &2 &1 &2 &2 &2 &1 &2 &2 &2 \\
2 &2 &2 &1 &2 &2 &2 &1 &2 &2 
\end{array}
\right)
$$

$$
A= \left(
\begin{array}{cccccccccc}
0 &1 &0 &0 &3 &3 &4 &4 &2 &2 \\
1 &0 &2 &3 &3 &0 &4 &1 &1 &4 \\
0 &2 &0 &3 &4 &4 &1 &1 &2 &1 \\
0 &3 &3 &0 &0 &1 &1 &2 &4 &2 \\
3 &3 &4 &0 &0 &1 &2 &2 &4 &0 \\
3 &0 &4 &1 &1 &0 &2 &0 &3 &4 \\
4 &4 &1 &1 &2 &2 &0 &0 &0 &3 \\
4 &1 &1 &2 &2 &0 &0 &0 &3 &3 \\
2 &1 &2 &4 &4 &3 &0 &3 &0 &0 \\
2 &4 &1 &2 &0 &4 &3 &3 &0 &0 
\end{array}
\right)
$$

\item %10
$$\mathbf{D}= \left(
\begin{array}{cccccccccc}
1 &3 &2 &2 &1 &2 &2 &2 &1 &2 \\
2 &1 &2 &2 &2 &1 &2 &2 &2 &2 \\
2 &3 &1 &2 &2 &2 &1 &2 &2 &1 \\
2 &2 &2 &1 &2 &2 &2 &1 &2 &2 
\end{array}
\right)
$$

$$
A= \left(
\begin{array}{cccccccccc}
0 &1 &0 &2 &3 &3 &4 &0 &2 &4 \\
1 &0 &2 &3 &3 &4 &4 &1 &1 &3 \\
0 &2 &0 &3 &4 &4 &1 &1 &2 &0 \\
2 &3 &3 &0 &0 &0 &1 &2 &4 &1 \\
3 &3 &4 &0 &0 &1 &0 &2 &4 &2 \\
3 &4 &4 &0 &1 &0 &2 &0 &3 &1 \\
4 &4 &1 &1 &0 &2 &0 &3 &0 &2 \\
0 &1 &1 &2 &2 &0 &3 &0 &3 &4 \\
2 &1 &2 &4 &4 &3 &0 &3 &0 &0 \\
4 &3 &0 &1 &2 &1 &2 &4 &0 &0 
\end{array}
\right)
$$

\item %11
$$\mathbf{D}= \left(
\begin{array}{cccccccccc}
1 &3 &2 &2 &1 &2 &2 &2 &1 &2 \\
3 &1 &2 &2 &2 &1 &2 &2 &2 &1 \\
2 &2 &1 &2 &2 &2 &1 &2 &2 &2 \\
2 &2 &2 &1 &2 &2 &2 &1 &2 &2 
\end{array}
\right)
$$

$$
A= \left(
\begin{array}{cccccccccc}
0 &1 &0 &2 &3 &3 &4 &4 &2 &2 \\
1 &0 &2 &3 &3 &4 &4 &1 &1 &0 \\
0 &2 &0 &3 &0 &4 &1 &1 &2 &4 \\
2 &3 &3 &0 &0 &0 &1 &2 &4 &1 \\
3 &3 &0 &0 &0 &1 &2 &2 &4 &4 \\
3 &4 &4 &0 &1 &0 &2 &0 &3 &1 \\
4 &4 &1 &1 &2 &2 &0 &0 &0 &3 \\
4 &1 &1 &2 &2 &0 &0 &0 &3 &3 \\
2 &1 &2 &4 &4 &3 &0 &3 &0 &0 \\
2 &0 &4 &1 &4 &1 &3 &3 &0 &0 
\end{array}
\right)
$$

\item %12
$$\mathbf{D}= \left(
\begin{array}{cccccccccc}
1 &3 &2 &2 &1 &2 &2 &2 &1 &2 \\
2 &1 &3 &2 &2 &1 &2 &2 &2 &1 \\
2 &2 &1 &2 &2 &2 &1 &2 &2 &2 \\
2 &2 &2 &1 &2 &2 &2 &1 &2 &2 
\end{array}
\right)
$$

$$
A= \left(
\begin{array}{cccccccccc}
0 &0 &0 &2 &3 &3 &4 &4 &2 &1 \\
0 &0 &2 &3 &3 &4 &4 &1 &1 &1 \\
0 &2 &0 &3 &4 &4 &1 &1 &2 &2 \\
2 &3 &3 &0 &0 &1 &1 &2 &0 &4 \\
3 &3 &4 &0 &0 &1 &2 &2 &4 &0 \\
3 &4 &4 &1 &1 &0 &2 &0 &3 &0 \\
4 &4 &1 &1 &2 &2 &0 &0 &0 &3 \\
4 &1 &1 &2 &2 &0 &0 &0 &3 &3 \\
2 &1 &2 &0 &4 &3 &0 &3 &0 &4 \\
1 &1 &2 &4 &0 &0 &3 &3 &4 &0 
\end{array}
\right)
$$

\item %13
$$\mathbf{D}= \left(
\begin{array}{cccccccccc}
1 &3 &2 &2 &1 &2 &2 &2 &1 &2 \\
2 &1 &2 &2 &2 &1 &2 &2 &2 &2 \\
2 &2 &1 &2 &2 &3 &1 &2 &2 &1 \\
2 &2 &2 &1 &2 &2 &2 &1 &2 &2 
\end{array}
\right)
$$

$$
A= \left(
\begin{array}{cccccccccc}
0 &0 &0 &2 &3 &3 &4 &4 &2 &1 \\
0 &0 &2 &3 &3 &4 &4 &1 &1 &1 \\
0 &2 &0 &3 &4 &4 &1 &1 &2 &0 \\
2 &3 &3 &0 &0 &1 &1 &2 &0 &4 \\
3 &3 &4 &0 &0 &1 &0 &2 &4 &2 \\
3 &4 &4 &1 &1 &0 &2 &0 &3 &3 \\
4 &4 &1 &1 &0 &2 &0 &3 &0 &2 \\
4 &1 &1 &2 &2 &0 &3 &0 &3 &0 \\
2 &1 &2 &0 &4 &3 &0 &3 &0 &4 \\
1 &1 &0 &4 &2 &3 &2 &0 &4 &0 
\end{array}
\right)
$$

\item %14
$$\mathbf{D}= \left(
\begin{array}{cccccccccc}
1 &3 &2 &2 &1 &2 &2 &2 &1 &2 \\
2 &1 &2 &2 &2 &1 &2 &2 &2 &2 \\
2 &2 &1 &3 &2 &2 &1 &2 &2 &1 \\
2 &2 &2 &1 &2 &2 &2 &1 &2 &2 
\end{array}
\right)
$$

$$
A= \left(
\begin{array}{cccccccccc}
0 &0 &0 &2 &3 &3 &4 &4 &2 &1 \\
0 &0 &2 &3 &3 &4 &4 &1 &1 &1 \\
0 &2 &0 &3 &4 &0 &1 &1 &2 &4 \\
2 &3 &3 &0 &0 &1 &1 &2 &4 &3 \\
3 &3 &4 &0 &0 &1 &0 &2 &4 &2 \\
3 &4 &0 &1 &1 &0 &2 &0 &3 &4 \\
4 &4 &1 &1 &0 &2 &0 &3 &0 &2 \\
4 &1 &1 &2 &2 &0 &3 &0 &3 &0 \\
2 &1 &2 &4 &4 &3 &0 &3 &0 &0 \\
1 &1 &4 &3 &2 &4 &2 &0 &0 &0 
\end{array}
\right)
$$

\end{enumerate}

\end{document}